\documentclass[11pt,a4paper]{article}

\usepackage{authblk}
\usepackage{caption2}%---change the caption of figures
\usepackage{CJK}
\usepackage{tabls}
\usepackage{bm}
\usepackage{multirow}
\usepackage{amssymb}
\usepackage{amsfonts}
\usepackage{mathrsfs}
\usepackage{empheq}
\usepackage{amsmath}
\usepackage{amsthm}
\usepackage{graphicx}
\usepackage{color}
\usepackage{indentfirst}
\usepackage{float}
\usepackage{cases}

\usepackage{setspace}
\theoremstyle{definition}
\newtheorem{definition}{Definition}[section]
\newtheorem{assumption}{Assumption}[section]

\numberwithin{equation}{section}
\theoremstyle{plain}
\newtheorem{theorem}{Theorem}[section]
\newtheorem{proposition}[theorem]{Proposition}
\newtheorem{lemma}[theorem]{Lemma}
\newtheorem{corollary}[theorem]{Corollary}

\theoremstyle{remark}
\newtheorem{remark}[theorem]{Remark}

\newcommand{\R}{\mathbb{R}}
\newcommand{\N}{\mathbb{N}}

\newcommand{\la}{\lambda}
\newcommand{\vp}{\varphi}

\newcommand{\angles}[1]{\langle #1 \rangle}

\date{\empty}

\begin{document}

\title{Wave front set of solutions to the fractional Schr\"{o}dinger equation}
\author[1]{Takumi Kanai\thanks{Email: 1124506@ed.tus.ac.jp}}
\author[2]{Ryo Muramatsu\thanks{Email: rmuramatsu@rs.tus.ac.jp}}
\author[2]{Yuusuke Sugiyama\thanks{Email: sugiyama.y@rs.tus.ac.jp}}

\affil[1]{Department of Mathematics, Graduate School of Science, Tokyo University of Science}
\affil[2]{Department of Mathematics, Faculty of Science, Tokyo University of Science}

\setlength{\baselineskip}{4.8mm}

\maketitle

\begin{abstract}
	In this paper,  we characterize the wave front sets of solutions to fractional Schr\"{o}dinger equations \(i\partial_{t}u =(-\Delta)^{\theta/2}u + V(x)u\)
	with $0<\theta <2$ via the wave packet transform (short-time Fourier transform). We clarify the relationship between the order \(\theta\) of the fractional Laplacian and the growth rate of the potential in the problem of propagation of singularities.
In particular, we present a theorem that bridges the propagation mechanisms of singularities for the Schr\"odinger and wave equations.

\end{abstract}

\section{Introduction}
In this article, we consider the initial value problem for the Schr\"{o}dinger equation
with the fractional Laplacian $(-\Delta)^{\theta/2}\ (0<\theta<2)$ and a potential $V(x)$,
\begin{equation}
  \begin{cases}
    i\partial_{t}u =(-\Delta)^{\theta/2}u + V(x)u,
      & (t,x)\in\mathbb{R}\times\mathbb{R}^{n}, \\
    u(0,x) = u_{0}(x),
      & x\in\mathbb{R}^{n}.
  \end{cases}
  \label{eq:main}
\end{equation}
Here $i=\sqrt{-1}$, $u:\mathbb{R}\times\mathbb{R}^{n}\to\mathbb{C}$, and
the fractional Laplacian $(-\Delta)^{\theta/2}$ is formally defined by
\[
  (-\Delta)^{\theta/2} f(x)
  = \mathcal{F}^{-1}\!\left(|\xi|^{\theta}\mathcal{F}[f](\xi)\right)(x),
\]
where $\mathcal{F}$ is the Fourier transform.
Moreover, the potential $V(x)$ is assumed to be a real-valued function.
The fractional Schr\"odinger equation arises as an extension of quantum mechanics when the Brownian paths in Feynman path integrals are replaced by L\'evy flights, yielding a fractional kinetic operator $(-\Delta)^{\theta/2}$ (see Laskin \cite{Laskin2002FractionalSE}). 
Moreover, Longhi \cite{Longhi2015FractionalSEOptics} proposed an optical realization in which the transverse light evolution in suitably designed optical cavities follows an effective fractional Schr\"odinger equation with a potential.

In recent years, fractional Schr\"odinger equations have begun to be studied extensively from various mathematical perspectives, including scattering theory, stationary problems, nonlinear dynamics, and Strichartz-type estimates (see, for instance, \cite{ZhangHuangZheng2021,FrankLenzmann2013,ChoHajaiejHwangOzawa2013,KohSeo2017,Gomez-Castro+Vazquez2019,AlphonseTzvetkov2025Circle}).
In this paper, we study the propagation of singularities of solutions in terms of the wave front set.
The wave front set is introduced  by H\"ormander to describe the precise information (position and direction) of the singularity of solutions to PDEs. In particular, H\"ormander has proved that 
 the singularities of solutions to hyperbolic PDEs propagate along null bi-characteristic curves (e.g. H\"ormander \cite{HormanderI, HormanderIII, Hormander1971FIO}).
Hyperbolic PDEs possess the so-called finite propagation property.
Consequently, the wave front set of their solutions can be characterized independently of the growth rate of lower-order terms, such as potential terms and the spatial decay rate of initial data.
When $\theta = 1$, the principal symbol of the equation in \eqref{eq:main}
coincides with that of the wave equation.
When $\theta=2$, the equation in \eqref{eq:main} is the Schr\"{o}dinger equation.
The propagation of singularities for solutions to the Schr\"odinger equation has been extensively investigated in the literature (e.g. Lascar \cite{Lascar1977}, Sakurai \cite{Sakurai1982QHWF}, Hassell and Wunsch \cite{HassellWunsch2005} and  Nakamura \cite{Nakamura2009WFSS, Nakamura2009}). 
Unlike hyperbolic equations, the Schr\"odinger evolution enjoys a dispersive smoothing effect, in the sense that the spatial decay of the initial data can yield improved regularity of the solution. 
As a result, the propagation of singularities behaves differently from the hyperbolic case.
For the Schr\"odinger setting, in Yajima \cite{Yajima1996}, it has been shown that the regularity of the solution (or  the fundamental solution) can change drastically depending on whether the potential is sub-quadratic or super-quadratic at infinity (see also Kato, Nakahashi and Tadano \cite{KatoNakahashiTadano2024}). 
The growth assumption on the potential imposed below is designed to bridge the hyperbolic and dispersive regimes.
Very recently,  Zhu \cite{Zhu2024GravityCapillary} established a characterization theorem for the wave front set of solutions to the fractional Schrödinger equation in the case without a potential term. The proof follows Nakamura’s approach, and Zhu further applied the theorem to the analysis of singularities for gravity–capillary water waves.

\begin{assumption}
	\label{ass:potential}
	The potential $V(x)$ is a real-valued function in $C^{\infty}(\mathbb{R}^{n})$.
	Furthermore, in the case $1 < \theta < 2$, there exists a constant $\nu < \frac{\theta}{\theta-1}$ such that for any multi-index $\alpha \in \mathbb{Z}_{\ge 0}^{n}$, there exists a constant $C_{\alpha} > 0$ satisfying
	\begin{equation*}
		|\partial_{x}^{\alpha}V(x)| \le C_{\alpha}\langle x\rangle^{\nu-|\alpha|}
	\end{equation*}
	for all $x \in \mathbb{R}^{n}$.
\end{assumption}

\if0
Before presenting the main theorem, we introduce the key concepts used in its statement: the wave front set and conic neighborhoods.
According to Hörmander \cite{Hormander}, the smoothness of a function is closely related to the decay rate of its Fourier transform at infinity. Specifically, a function $u \in \mathcal{S}'(\mathbb{R}^n)$ belongs to the Sobolev space $H^s(\mathbb{R}^n)$ if and only if
\begin{equation*}
	\int_{\mathbb{R}^n} (1+|\xi|^2)^s |\hat{u}(\xi)|^2 d\xi < \infty
\end{equation*}
holds. This implies that $\hat{u}(\xi)$ decays at the rate of $|\xi|^{-s}$ at infinity. However, such a concept of global smoothness cannot capture the local singularities of a function. Therefore, Hörmander introduced the concept of the wave front set. The definition is as follows.
\fi

\begin{definition}[Wave front set] \label{def:wave_front_set}
	For $f \in \mathcal{S}'(\mathbb{R}^n)$, we say that $(x_0, \xi_0) \in \mathbb{R}^n \times (\mathbb{R}^n \setminus \{0\})$ is not contained in the wave front set $WF(f)$ of $f$ (denoted by $(x_0, \xi_0) \notin WF(f)$) if there exists a function $\chi \in C_0^\infty(\mathbb{R}^n)$ with $\chi(x_0) \neq 0$ and a conic neighborhood $\Gamma$ of $\xi_0$ such that for any $N \in \mathbb{N}$, there exists a constant $C_N > 0$ satisfying
	\begin{equation*}
		|\mathcal{F}[\chi f](\xi)| \le C_N(1+|\xi|)^{-N} \quad (\text{for all } \xi \in \Gamma).
	\end{equation*}
	Here, $\Gamma$ is called a conic neighborhood of $\xi_0$ if $\Gamma$ is an open neighborhood of $\xi_0$ in $\mathbb{R}^n \setminus \{0\}$ and $\alpha \xi \in \Gamma$ holds for any $\xi \in \Gamma$ and $\alpha > 0$.
\end{definition}
Our main result describes the singularities of solutions to the fractional Schr\"odinger equation in terms of the framework introduced by Kato, Kobayashi and Ito \cite{KatoKobayashiIto2017}.
To this end, we employ the wave packet transform originally introduced by C\'ordoba and Fefferman \cite{Cordoba1978}.

\begin{definition}[Wave Packet Transform] \label{def:wave_packet_transform}
	For $\varphi \in \mathcal{S}(\mathbb{R}^n) \setminus \{0\}$ (basic wave packet) and $f \in \mathcal{S}'(\mathbb{R}^n)$, we define the wave packet transform $W_\varphi[f](x, \xi)$ by
	\begin{equation*}
		W_\varphi[f](x, \xi) = \int_{\mathbb{R}^n} \overline{\varphi(y-x)} f(y) e^{-iy\cdot\xi} dy.
	\end{equation*}
\end{definition}

The function $\varphi$ appearing in the wave packet transform is referred to as a wave packet or a window function. 
Moreover, for scaled wave packets, we often use the following notation for a parameter $b$ satisfying $0 < b < 1$ and $\lambda \ge 1$:
	\begin{equation*}
		\varphi_\lambda(x) = \lambda^{nb/2} \varphi(\lambda^b x).
	\end{equation*}
The equivalence of the following definitions of the wave front set is established by Kato, Kobayashi, and Ito, and it will be used in the proof of the main theorem.
\begin{proposition}[Kato, Kobayashi and Ito \cite{KatoKobayashiIto2017}] \label{prop:wfs_characterization}
	Let $(x_0, \xi_0) \in \mathbb{R}^n \times (\mathbb{R}^n \setminus \{0\})$ and $f \in \mathcal{S}'(\mathbb{R}^n)$. Fix $0 < b < 1$. Then, the following conditions are equivalent:
	\begin{itemize}
		\item[(i)] $(x_0, \xi_0) \notin WF(f)$
		\item[(ii)] There exists a conic neighborhood $V$ of $(x_0, \xi_0)$ such that for any $N \in \mathbb{N}$, any $a \ge 1$, and any $\varphi \in \mathcal{S}(\mathbb{R}^n) \setminus \{0\}$, there exists a constant $C_{N,a,\varphi} > 0$ satisfying
		      \begin{equation*}
			      |W_{\varphi_\lambda}[f](x, \lambda\xi)| \le C_{N,a,\varphi} \lambda^{-N}
		      \end{equation*}
		      for all $(x, \xi) \in V$ with $\lambda \ge 1$ and $a^{-1} \le |\xi| \le a$.
	\end{itemize}
\end{proposition}

The following theorem is our main result, which provides a characterization of the wave front set $WF(u(t))$ for the fractional Schr\"{o}dinger equation via the wave packet transform.

\begin{theorem}
	\label{thm:main}
	Suppose Assumption \ref{ass:potential} holds.
	Let $u_0 \in L^2(\mathbb{R}^n)$ and let $u(t,x)$ be the solution to \eqref{eq:main} in the class $C(\mathbb{R}; L^2(\mathbb{R}^n))$.
	Let $b$ be a parameter satisfying $0 < b < (2-\theta)/2$.
	Then, $(x_0, \xi_0) \notin WF(u(t))$ if and only if there exists a conic neighborhood $V = K \times \Gamma$ of $(x_0, \xi_0)$
	such that for all $N \in \mathbb{N}$, all $a \ge 1$, and all $\varphi \in \mathcal{S}(\mathbb{R}^n) \setminus \{0\}$,
	there exist constants $\lambda_0 > 0$ and $C_{N,a,\varphi} > 0$ satisfying
	\begin{equation}
		|W_{\varphi_{\lambda}}[u_0](x(0; t, x, \lambda\xi), \xi(0; t, x, \lambda\xi))| \le C_{N,a,\varphi}\lambda^{-N}
	\end{equation}
	for all $\lambda > \lambda_0$ and $(x, \xi) \in V$ with $a^{-1} \le |\xi| \le a$, where,$(x(s), \xi(s)) := (x(s; t, x, \lambda\xi), \xi(s; t, x, \lambda\xi))$ is the solution to the following initial-value problem of Hamilton equations:
	\begin{equation}
		\begin{cases}
			\frac{d}{ds}x(s) = \theta |\xi(s)|^{\theta-2}\xi(s), & x(t)=x,            \\
			\frac{d}{ds}\xi(s) = -\nabla V(x(s)),                 & \xi(t)=\lambda\xi.
		\end{cases}
		\label{eq:hamilton}
	\end{equation}
\end{theorem}

As shown by the following corollary, there exists a decay rate of the potential such that the characteristic curves in Theorem~\ref{thm:main} reduce to those of the potential-free system.
This condition corresponds to a short-range assumption on the potential in the study of the propagation of singularities for solutions to the Schr\"odinger equation.

\begin{corollary}\label{wfs:cor:main_no_potential}
	In the case $1<\theta<2$, assume that there exists $\nu<\frac{\theta}{2(\theta-1)}$ such that for any multi-index $\alpha\in \mathbb{Z}_{\geq 0}^n$, there exists a constant $C_\alpha>0$ satisfying
	$$
		|\partial_x^\alpha V(x)|\le C_\alpha\langle x\rangle^{\nu-|\alpha|},\quad  x\in\R^n.
	$$
	Let $b$ satisfy $(\theta-1)(\nu-1)<b<(2-\theta)/2$.
	Then, the assertions (i) and (ii) of Theorem \ref{thm:main} are equivalent, where $(\tilde{x}(s),\tilde{\xi}(s))$  is a solution to the following initial-value problem of the potential-free Hamilton equation
	\begin{equation}\label{wfs:eq:characteristics_ode_no_potential}
		\begin{cases}
			\frac{d}{ds}\tilde{x}(s)=\theta|\tilde{\xi}(s)|^{\theta-2}\tilde{\xi}(s), & \tilde{x}(t)=x,            \\
			\frac{d}{ds}\tilde{\xi}(s)=0,                                             & \tilde{\xi}(t)=\lambda\xi.
		\end{cases}
	\end{equation}
\end{corollary}
\begin{remark}
Corollary \ref{wfs:cor:main_no_potential} holds for $0 <\theta \le1$ without any decay condition on $V$.
\end{remark}
Furthermore, in the case $0<\theta\le1$, the following corollary holds, which characterizes the wave front set using the wave front set of initial data.

\begin{corollary}\label{wfs:cor:main_time_propagation}
	Let $u(t,x)$ be the solution to \eqref{eq:main} in $C(\mathbb{R};L^{2}(\mathbb{R}^{n}))$, and assume that the potential $V(x)$ satisfies the conditions in Corollary \ref{wfs:cor:main_no_potential}.
	Then, if the exponent of the fractional Laplacian is $\theta =1$, we have
	\begin{equation}
		WF(u(t)) =\chi_{t,0} (WF(u_0)),
	\end{equation}
	where $\chi_{t,0}$ is the mapping defined by using the solution to \eqref{wfs:eq:characteristics_ode_no_potential} as follows:
	\begin{equation*}
		\chi_{\tau,t}(x,\xi) = (\tilde{x}(\tau;t,x,\xi), \tilde{\xi}(\tau;t,x,\xi)).
	\end{equation*}
	If $0<\theta<1$, we have
	\begin{equation}
		WF(u(t)) = WF(u_0).
	\end{equation}
\end{corollary}
\begin{remark}
In the main theorems, although the theorem is described using different Hamiltonian flows for the cases $1 \leq \theta < 2$, $\theta = 1$, and $\theta < 1$, this does not lead to any contradiction. 
Indeed, the corresponding Hamiltonian flows coincide asymptotically for large  $\lambda$.
\end{remark}

We review some relations between our main theorems and known results.
In Nakamura \cite{Nakamura2009}, for variable-coefficient Schr\"odinger equations with sub-quadratic potentials, the wave front set of solutions is characterized in terms of that of the initial data by means of pseudo-differential operators and the associated Hamiltonian flow.
In Nakamura \cite{Nakamura2009WFSS}, for the same equation, the wave front set is characterized by $WF(e^{tH_0} u_0)$ where $H_0$ is the free Schr\"odinger operator. In Kato, Kobayashi and Ito \cite{KatoKobayashiIto2017} and Kato and Ito \cite{wfs:KatoIto2014},  the wave front set for Schr\"odinger equations with potentials is characterized by means of the wave packet transform, where the window function is taken to be a solution to the free Schr\"odinger equation with a rapidly decaying initial data. 
In contrast, in our theorem, the wave packet (the window function) does not involve solutions to the corresponding free equation.
This illustrates the difference in the propagation of singularities between the Schr\"odinger equation and the fractional Schr\"odinger equation, corresponding respectively to the cases $\theta=2$ and $\theta<2$.
The proof is based on the method developed by Kato, Kobayashi, and Ito \cite{KatoKobayashiIto2017}. 
More precisely, their approach allows one to reduce the fractional Schr\"odinger equation to a transport equation, and to obtain a representation of the solution by means of the method of characteristics (see also \cite{ItoKatoKobayashi2012}). 
The argument is completed by an induction on the decay in the parameter $\lambda$, where the representation formula is used to derive the corresponding decay estimates at each step.
However, since the symbol of the fractional Laplacian is singular at the origin, the proof requires a careful decomposition into a neighborhood of the origin and its complement. 
The key idea is to estimate these two regions separately. 
In particular, it is essential to construct the characteristic curves (Hamiltonian flows) so that they do not reach the origin. In this construction, unlike the Schr\"odinger case, the momentum equation contains the spatial derivative of the potential and therefore may grow in the $x$-variable. 
However, under the  assumption on the potential, one can still show that the momentum $\xi (t)$ grows at most linearly with respect to $\lambda$.

\subsection*{Notations}
Throughout this paper, we may use different notations depending on the context.
For clarity, we will often use the simplified notation $W_{\varphi}u(t, x, \xi)$ to denote $W_{\varphi}[u(t, \cdot)](x, \xi)$ when $u$ is the solution to \eqref{eq:main}.
Throughout the paper, $C$ and $C_j$ denote generic positive constants whose values may change from line to line. 
Constants depending on parameters such as $a$ or $b$ are sometimes denoted by $C_a$ or $C_{a,b}$.

The Fourier transform of a function $f$ is defined by
\[
  \mathcal{F}[f](\xi)
  = \int_{\mathbb{R}^n} e^{-i x \cdot \xi} f(x)\, dx,
\]
and its inverse Fourier transform by
\[
  \mathcal{F}^{-1}[g](x)
  = (2\pi)^{-n} \int_{\mathbb{R}^n} e^{i x \cdot \xi} g(\xi)\, d\xi.
\]

We also use the standard Japanese bracket notation
\[
  \langle x \rangle = (1+|x|^2)^{1/2},
\]
and similarly for $\langle \xi \rangle$.

\section{Estimates for characteristic curves}
In this section, as a preparation for the proof of the main theorem, we discuss the existence and the order estimates with respect to $\lambda$ of the solutions $(x(s), \xi(s))$ of \eqref{eq:hamilton}.

\begin{lemma}
	\label{lem:estimates}
	Let $ T < \infty$, $x \in K$ for a compact set of $\R^n$ and $a^{-1} \le  |\xi| \le a$ for $a>0$.
	Under Assumption \ref{ass:potential}, the solutions $(x(s), \xi(s))$ of \eqref{eq:hamilton} satisfy the following estimates for $|s-t| \le T$.
	There exist positive constants $C_{1}, C_{2}, C_{3}, C_{4}$, and $\lambda_0$ depending only on $a$, $\theta$, $K$ and $T$, but independent of $\lambda$ such that for all $\lambda \ge \lambda_0$, we have
	\begin{equation} \label{xi-lam-es}
		C_{2} \lambda \le |\xi(s)| \le C_{1} \lambda.
	\end{equation}
	Furthermore, for $x(s)$, we have
	\begin{equation} \label{x-lam-es}
		\begin{cases}
			C_{4} \lambda^{\theta-1} \le |x(s)| \le C_{3} \lambda^{\theta-1}, & \text{if } 1 < \theta < 2,   \\
			|x(s)| \le C_{3},                                                  & \text{if } 0 < \theta \le 1.
		\end{cases}
	\end{equation}
\end{lemma}

\begin{proof}
	We prove the lemma using Picard's iteration method.
	We define the sequence of approximations as follows:
	\begin{equation*}
		\begin{cases}
			x^{(0)}(s) = x + \int_t^s \theta |\lambda \xi|^{\theta-2} \lambda \xi \, d\tau,           \\
			\xi^{(0)}(s) = \lambda \xi,                                                               \\
			x^{(N+1)}(s) = x + \int_t^s \theta |\xi^{(N)}(\tau)|^{\theta-2} \xi^{(N)}(\tau) \, d\tau, \\
			\xi^{(N+1)}(s) = \lambda \xi - \int_t^s \nabla_x V(x^{(N)}(\tau)) \, d\tau.
		\end{cases}
	\end{equation*}
	Set $C_K:=\sup_{x\in K}|x|$.
We show that \eqref{xi-lam-es} and \eqref{x-lam-es} are true for $\xi^{(N)}$ and $x^{(N)}$ for any $N$ by induction. From the assumption $a^{-1} \leq |\xi| \leq a$, we have
 	\begin{equation*}
		a^{-1}\lambda \leq |\xi^{(0)}(s)| \leq a \lambda,
	\end{equation*}
	which implies \eqref{xi-lam-es}  for $N=0$.
	From the definition of $x^{(0)}(s)$, we obtain the following inequality:
	\begin{equation*}
		\begin{aligned}
			|x^{(0)}(s)| & = \left|x+ \int_t^s \theta |\lambda \xi|^{\theta-2} \lambda \xi \, d\tau\right| \\
			             & \le |x| + |s-t| \theta |\lambda \xi|^{\theta-1}                                 \\
			             & = |x| + T \theta \lambda^{\theta-1} |\xi|^{\theta-1}.
		\end{aligned}
	\end{equation*}
	Note that $|x|\le C_K$ and $|\xi|^{\theta-1} \le \max\{a^{\theta-1}, a^{1-\theta}\}$. Hence
	\begin{equation*}
		|x^{(0)}(s)| \le C_K + T\theta \max\{a^{\theta-1}, a^{1-\theta}\}\,\lambda^{\theta-1}
		\le
		\begin{cases}
			C_{3}\lambda^{\theta-1} & \text{if } 1 < \theta < 2,   \\
			C_{3}                  & \text{if } 0 < \theta \le 1.
		\end{cases}
	\end{equation*}
	Thus, the upper estimate for $x^{(0)}(s)$ holds for $N=0$.
	Next, assuming that the inequalities in Lemma \ref{lem:estimates} hold for $N$, we show that they hold for $N+1$.
	First, we consider the estimate for $\xi^{(N+1)}(s)$. The following inequality holds:
	\begin{equation*}
		\begin{aligned}
			|\xi^{(N+1)}(s)| & \le a\lambda + \int_t^s |\nabla_x V(x^{(N)}(\tau))| d\tau                 \\
			                 & \le a\lambda + T  \sup_{|s-t|\leq T} |\nabla_x V(x^{(N)}(s))|             \\
			                 & \le a\lambda + T C \sup_{|s-t|\leq T} \langle x^{(N)}(s) \rangle^{\nu-1}.
		\end{aligned}
	\end{equation*}
	Here, we consider the upper estimate in the case $0 < \theta \leq 1$.
	\\
	When $\nu \leq 1$, since $\sup_{|s-t|\leq T} \langle x^{(N)}(s) \rangle^{\nu-1} \le 1$, we have
	\begin{equation*}
		|\xi^{(N+1)}(s)| \le a\lambda + T C \le (a + T C) \lambda.
	\end{equation*}
	Thus, taking $\lambda$ sufficiently large, the upper estimate holds in this case.
	Next, we consider the upper estimate in the case $1 < \theta < 2$.
	Regarding the case $\nu \leq 1$, since $\sup_{|s-t|\leq T} \langle x^{(N)}(s) \rangle^{\nu-1} \le 1$, we have
	\begin{equation*}
		|\xi^{(N+1)}(s)| \le a\lambda + T C \le (a + T C) \lambda.
	\end{equation*}
	Thus, the upper estimate holds in this case.
	Regarding the case $\nu > 1$, using the inductive hypothesis and $\langle x \rangle \leq 1 + |x|$, we have
	\begin{equation*}
		\begin{aligned}
			\sup_{|s-t|\leq T} \langle x^{(N)}(s) \rangle^{\nu-1} & \le \sup_{|s-t|\leq T}(1+|x^{(N)}(s)|)^{\nu-1} \\
			                                                      & \leq(1+C_{3} \lambda^{\theta-1})^{\nu-1}         \\
			                                                      & =(1+C_{3})^{\nu-1}\lambda^{(\theta-1)(\nu-1)}.
		\end{aligned}
	\end{equation*}
	Therefore, since the assumption $\nu<\frac{\theta}{\theta-1}$, we obtain
	\begin{equation*}
		\begin{aligned}
			|\xi^{(N+1)}(s)| & \le a\lambda + T C (1+C_{3})^{\nu-1} \lambda^{(\theta-1)(\nu-1)}                   \\
			                 & \le \left(a + T C (1+C_{3})^{\nu-1} \lambda^{(\theta-1)(\nu-1)-1} \right) \lambda.
		\end{aligned}
	\end{equation*}
	Thus, the upper estimate holds in this case as well.
	From the above, the upper estimate for $\xi^{(N+1)}(s)$ holds when $1 < \theta < 2$.
	Next, we consider the lower estimate for $\xi^{(N+1)}(s)$ in the case $0 < \theta \le 1$.
	We have
	\begin{equation*}
		\begin{aligned}
			|\xi^{(N+1)}(s)| & \ge a^{-1}\lambda - \int_t^s |\nabla_x V(x^{(N)}(\tau))| d\tau                 \\
			                 & \ge a^{-1}\lambda - TC  \sup_{|s-t|\leq T} \langle x^{(N)}(s) \rangle^{\nu-1}.
		\end{aligned}
	\end{equation*}
	When $\nu \le 1$, we have
	\begin{equation*}
		|\xi^{(N+1)}(s)| \ge a^{-1}\lambda - TC
		= (a^{-1} - TC \lambda^{-1}) \lambda.
	\end{equation*}
	Thus, taking $\lambda$ sufficiently large, the lower estimate for $\xi^{(N+1)}(s)$ holds.
	When $\nu > 1$, we obtain
	\begin{equation*}
		\begin{aligned}
			\sup_{|s-t|\leq T} \angles{x^{(N)}(s)}^{\nu-1} \le (1+C_{3})^{\nu-1}\la^{(\theta-1)(\nu-1)}.
		\end{aligned}
	\end{equation*}
	Therefore,
	\begin{align*}
		|\xi^{(N+1)}(s)| & \ge a^{-1}\lambda - T C (1+C_{3})^{\nu-1} \lambda^{(\theta-1)(\nu-1)}                 \\
		                 & = \left(a^{-1} - T C (1+C_{3})^{\nu-1} \lambda^{(\theta-1)(\nu-1)-1} \right) \lambda.
	\end{align*}
	Hence, taking $\lambda$ sufficiently large, since $(\theta -1)(\nu -1) -1 <0$, the lower estimate for $\xi^{(N+1)}(s)$ holds.
	From the above, the lower estimate for $\xi^{(N+1)}(s)$ holds for the case $0 < \theta \le 1$.
	Next, we consider the estimate for $x^{(N+1)}(s)$. The following inequality holds:
	\begin{equation*}
		\begin{aligned}
			|x^{(N+1)}(s)| & = \left|x + \int_t^s \theta |\xi^{(N)}(\tau)|^{\theta-2} \xi^{(N)}(\tau) \, d\tau\right| \\
			               & \le |x| + \theta \int_t^s |\xi^{(N)}(\tau)|^{\theta-1} d\tau                             \\
			               & \le C_K + \theta \, T \sup_{|s-t|\leq T} |\xi^{(N)}(\tau)|^{\theta-1}.
		\end{aligned}
	\end{equation*}
	In the case $0 < \theta \leq 1$, from the lower estimate of $\xi^{(N)}(s)$ and the fact that $\lambda^{\theta-1} \le 1$, we have
	\begin{equation*}
		\begin{aligned}
			|x^{(N+1)}(s)| & \le C_K + \theta \, T (C_{2} \lambda)^{\theta-1} \\
			               & \le C_K + \theta \, T (C_{2})^{\theta-1}.
		\end{aligned}
	\end{equation*}
	Thus, the upper estimate holds.
	Next, we consider the upper estimate in the case $1 < \theta < 2$.
	From the upper estimate of $\xi^{(N)}(s)$, we have for sufficiently large $\lambda_0$
	\begin{equation*}
		\begin{aligned}
			|x^{(N+1)}(s)| & \le C_K + \theta \, T (C_{1} \lambda)^{\theta-1}             \\
			               & \le (C_K + \theta \, T C_{1}^{\theta-1}) \lambda^{\theta-1}.
		\end{aligned}
	\end{equation*}
	Thus, the upper estimate holds.
	This completes the induction, showing that the order estimates for the sequence of approximations $(x^{(N)}(s),\xi^{(N)}(s))$ are satisfied for any $N$.

	Next, using the order estimates for the sequence of approximations, we show that the solution $(x(s), \xi(s))$ to equation \eqref{eq:hamilton} is the limit of the sequence.
	We show this for the case $1 < \theta < 2$. The case $0 < \theta \le 1$ can be shown similarly.
	Hereinafter, we fix an arbitrary $\lambda \ge \lambda_0$.

	Define the region $S_\lambda$ by
	\begin{equation*}
		\begin{aligned}
		S_\lambda := \{ (s,x,\xi) \in \mathbb{R} \times \mathbb{R}^n \times \mathbb{R}^n :
		&  |s-t| \leq T,\\
		&  |x| \leq C_{3} \lambda^{\theta-1},\\
		& C_{2} \lambda \leq |\xi| \leq C_{1} \lambda \}.
		\end{aligned}
	\end{equation*}
	$S_\lambda$ is a bounded closed region. Also, from the order estimates for the sequence of approximations, $(s, x^{(N)}(s), \xi^{(N)}(s)) \in S_\lambda$ holds for any $N$.
	Define the function $F(s,x,\xi)$ as follows:
	\begin{equation*}
		F(s,x,\xi) := (\theta |\xi|^{\theta-2} \xi, -\nabla_x V(x)).
	\end{equation*}
We can easily check that $F$ is Lipschitz continuous on $S_{\lambda}$.
	It follows from the standard Picard--Lindel\"of theorem (applied on the compact interval $\{s:  |s-t|\le T\}$ and the closed set $S_{\lambda}$) that the integral equation
	\begin{equation*}
		z(s)=(x,\lambda\xi)+\int_{t}^{s} F(\tau,z(\tau))\,d\tau
	\end{equation*}
	admits a unique solution $z(s)=(x(s),\xi(s))$ with $(s,x(s),\xi(s))\in S_{\lambda}$, and that the Picard iterates $z^{(N)}(s)$ converge to $z(s)$ uniformly on $ |s-t|\le T$.
	Passing to the limit $N\to\infty$ in the inductive estimates for $(x^{(N)}(s),\xi^{(N)}(s))$ yields \eqref{xi-lam-es} and \eqref{x-lam-es} for $(x(s),\xi(s))$.
\end{proof}

\section{Representation of solutions via wave packet transform}
In this section, following Kato, Ito, and Kobayashi \cite{wfs:KatoIto2014}, we derive a representation formula for solutions via the wave packet transform.
We apply the wave packet transform $W_{\varphi_{\lambda}}$ to the equation in \eqref{eq:main}.
\begin{equation*}
	W_{\varphi_{\lambda}} [i\partial_{t}u] - W_{\varphi_{\lambda}}[(-\Delta)^{\theta/2}u] - W_{\varphi_{\lambda}}[V(x)u] = 0
\end{equation*}
First, for the time derivative term, we have
\begin{equation*}
	W_{\varphi_{\lambda}} [i\partial_{t}u](t, x, \xi) = i\partial_{t}W_{\varphi_{\lambda}} [u](t, x, \xi).
\end{equation*}
Next, we perform the calculation for the fractional Laplacian term.
\begin{multline}
W_{\varphi_{\lambda}}[(-\Delta)^{\theta/2}u](t, x, \xi)
= \int_{\mathbb{R}^n} \overline{\varphi_{\lambda}(y-x)}\, ((-\Delta)^{\theta/2}u(t, \cdot))(y)\, e^{-iy\cdot\xi}\, dy \\
= \int_{\mathbb{R}^n} \overline{\varphi_{\lambda}(y-x)}\,
\Bigl( \int_{\mathbb{R}^n} e^{iy\cdot\eta}|\eta|^{\theta}\widehat{u}(t, \eta)\, d\eta \Bigr)
 e^{-iy\cdot\xi}\, dy \\
= \int_{\mathbb{R}^n} \overline{\varphi_{\lambda}(y-x)}\,
\Bigl( \int_{\mathbb{R}^n} e^{iy\cdot\eta}(1-\chi(\eta))|\eta|^{\theta}\widehat{u}(t, \eta)\, d\eta \Bigr)
 e^{-iy\cdot\xi}\, dy \\
\quad + \int_{\mathbb{R}^n} \overline{\varphi_{\lambda}(y-x)}\,
\Bigl( \int_{\mathbb{R}^n} e^{iy\cdot\eta}\chi(\eta)|\eta|^{\theta}\widehat{u}(t, \eta)\, d\eta \Bigr)
 e^{-iy\cdot\xi}\, dy, 
\label{eq:frac_laplacian_expansion}
\end{multline}
where $\chi \in C^\infty (\R^n)$ is a cut-off function defined as follows:
\begin{equation*}
	\chi(\eta) =
	\begin{cases}
		1                                      & (|\eta| \le C_2\lambda_0/2),                   \\
		\text{smoothly decreasing from 1 to 0} & (C_2\lambda_0/2 < |\eta| < C_2\lambda_0), \\
		0                                      & (|\eta| \ge C_2\lambda_0)
	\end{cases}
\end{equation*}
and $\lambda_0$ and $C_2$ are constants satisfying the assertion of Lemma \ref{lem:estimates}.
We put  $a(\eta) = \chi(\eta)|\eta|^{\theta}$, $b(\eta) = (1-\chi(\eta))|\eta|^{\theta}$. 
Since $b(\eta)$ is a $C^{\infty}$ function, applying Taylor expansion to $b$ around $\eta=\xi$, we have
\begin{equation}
	b(\eta) = b(\xi) + \nabla_{\xi}b(\xi) \cdot (\eta-\xi) + R_{b}(\eta, \xi). \label{decomp-b}
\end{equation}
Here the remainder term $R_{b}(\eta, \xi)$ is decomposed as $R_{b}(\eta, \xi) := R_{b,1}(\eta, \xi) + R_{b,2}(\eta, \xi)$, defined respectively as follows:
\begin{align*}
	R_{b,1}(\eta, \xi) & := \sum_{2 \le |\alpha| < L} \frac{1}{\alpha!} (\partial_{\xi}^{\alpha}b)(\xi)(\eta-\xi)^{\alpha},                                           \\
	R_{b,2}(\eta, \xi) & := L \sum_{|\alpha|=L} \frac{(\eta-\xi)^{\alpha}}{\alpha!} \int_{0}^{1} (1-\tau)^{L-1} (\partial_{\xi}^{\alpha}b)(\xi+\tau(\eta-\xi)) d\tau.
\end{align*}
\eqref{eq:frac_laplacian_expansion} can be written as
\begin{equation}
	W_{\varphi_{\lambda}} [(-\Delta)^{\theta/2}u](t, x, \xi) = W_{\varphi_{\lambda}} [a(D)u](t, x, \xi) + W_{\varphi_{\lambda}} [b(D)u](t, x, \xi).
	\label{eq:symbol_decomposition}
\end{equation}
Applying \eqref{decomp-b} to the second term $W_{\varphi_{\lambda}} [b(D)u]$ in \eqref{eq:symbol_decomposition}, we consider the term
\begin{equation*}
	\nabla_{\xi}b(\xi) \cdot \iint \overline{\varphi_{\lambda}(y-x)} e^{iy\cdot(\eta-\xi)} (\eta-\xi) \widehat{u}(\eta) d\eta dy.
\end{equation*}
Using the relation $(\eta-\xi)e^{iy\cdot(\eta-\xi)} = -i\nabla_{y}e^{iy\cdot(\eta-\xi)}$, from the integration by parts with respect to $y$, we have
\begin{equation*}
	\begin{aligned}
		\iint \overline{\varphi_{\lambda}(y-x)} e^{iy\cdot(\eta-\xi)} (\eta-\xi) \widehat{u}(\eta) d\eta dy
		 & = -i \iint \overline{\varphi_{\lambda}(y-x)} \nabla_{y} e^{iy\cdot(\eta-\xi)} \widehat{u}(\eta) d\eta dy   \\
		 & = i \iint \nabla_{y} \overline{\varphi_{\lambda}(y-x)} e^{iy\cdot(\eta-\xi)} \widehat{u}(\eta) d\eta dy    \\
		 & = i \iint (-\nabla_{x}) \overline{\varphi_{\lambda}(y-x)} e^{iy\cdot(\eta-\xi)} \widehat{u}(\eta) d\eta dy \\
		 & = -i \nabla_{x} W_{\varphi_{\lambda}} [u](x, \xi).
	\end{aligned}
\end{equation*}
Thus, the first-order term is $-i \nabla_{\xi}b(\xi) \cdot \nabla_{x} W_{\varphi_{\lambda}} [u](x, \xi)$. Therefore,
\begin{align*}
	W_{\varphi_{\lambda}} [b(D)u](x, \xi)
	&= (b(\xi) - i \nabla_{\xi}b(\xi) \cdot \nabla_{x}) W_{\varphi_{\lambda}} [u](x, \xi) \\
	&\quad + \iint \overline{\varphi_{\lambda}(y-x)} e^{iy\cdot(\eta-\xi)} R_{b}(\eta, \xi)
	\widehat{u}(\eta)\, d\eta\, dy
\end{align*}
holds.
Furthermore, we apply the Taylor expansion to the potential term $V(y)$  with the order $L$ around $y=x$
\begin{equation*}
	V(y) = V(x) + \nabla V(x) \cdot (y-x) + R_{V}(y, x),
\end{equation*}
where $R_{V}(y, x) := R_{V,1}(y, x) + R_{V,2}(y, x)$ and 
\begin{align*}
	R_{V,1}(y, x) & := \sum_{2 \le |\alpha| < L} \frac{1}{\alpha!} (\partial_{x}^{\alpha}V)(x)(y-x)^{\alpha},                                      \\
	R_{V,2}(y, x) & := L \sum_{|\alpha|=L} \frac{(y-x)^{\alpha}}{\alpha!} \int_{0}^{1} (1-\tau)^{L-1} (\partial_{x}^{\alpha}V)(x+\tau(y-x)) d\tau.
\end{align*}
For the term  $\nabla V(x) \cdot \int (y-x) \overline{\varphi_{\lambda}(y-x)} u(y) e^{-iy\cdot\xi} dy$, using the relation $ye^{-iy\cdot\xi} = i\nabla_{\xi}e^{-iy\cdot\xi}$, we can rewrite it as follows:
\begin{align*}
	\nabla V(x) \cdot \int (y-x)\, \overline{\varphi_{\lambda}(y-x)}\, u(y)\, e^{-iy\cdot\xi}\, dy
	&= i (\nabla V(x) \cdot \nabla_{\xi}) W_{\varphi_{\lambda}} [u](x, \xi) \\
	&\quad - (x \cdot \nabla V(x)) W_{\varphi_{\lambda}} [u](x, \xi).
\end{align*}
Therefore, the potential term is expressed as
\begin{equation*}
	W_{\varphi_{\lambda}} [V(\cdot)u] = (V(x) - x \cdot \nabla V(x)) W_{\varphi_{\lambda}} [u] + i (\nabla V(x) \cdot \nabla_{\xi}) W_{\varphi_{\lambda}} [u] + W_{\varphi_{\lambda}} [R_{V}u],
\end{equation*}
where $W_{\varphi_{\lambda}} [R_{V}u] = W_{\varphi_{\lambda}} [R_{V,1}u] + W_{\varphi_{\lambda}} [R_{V,2}u]$.
Collecting the above results, we obtain the following first-order linear transport equation for the wave packet transform $W_{\varphi_{\lambda}}u$.
\begin{equation}
\begin{aligned}
	&(\partial_{t} + \nabla_{\xi}b(\xi) \cdot \nabla_{x} - \nabla_{x}V(x) \cdot \nabla_{\xi})\, W_{\varphi_{\lambda}} u(t, x, \xi) \\
	&\qquad = i P(x, \xi)\, W_{\varphi_{\lambda}} u(t, x, \xi) + i Ru(t, x, \xi),
\end{aligned}
	\label{eq:transport_equation}
\end{equation}
where $P(x, \xi) = -b(\xi) - V(x) + x \cdot \nabla V(x)$ and
\begin{align}
	Ru(t, x, \xi) & = - W_{\varphi_{\lambda}} [a(D)u(t, \cdot)](x, \xi)          \notag                                                      \\
	              & \quad - \iint \overline{\varphi_{\lambda}(y-x)} e^{iy\cdot(\eta-\xi)} R_{b}(\eta, \xi) \widehat{u}(\eta) d\eta dy \notag \\
	              & \quad - W_{\varphi_{\lambda}} [R_{V}u(t, \cdot)](x, \xi).
	\label{eq:remaindeRu}
\end{align}
In what follows, we derive a representation of the solution to \eqref{eq:transport_equation} by the method of characteristics.
Noting that Lemma \ref{lem:estimates} holds when $\lambda \ge \lambda_0$, the characteristic curves of \eqref{eq:transport_equation} are solutions to \eqref{eq:hamilton}.
Therefore, $W_{\varphi_{\lambda}}u(s, x(s), \xi(s))$ satisfies that
\begin{equation*}
	\frac{d}{ds}W_{\varphi_{\lambda}}u(s, x(s), \xi(s)) = i P(x(s), \xi(s)) W_{\varphi_{\lambda}}u(s, x(s), \xi(s)) + i Ru(s, x(s), \xi(s)),
\end{equation*}
which yields that the following integral equation holds
\begin{equation}
	\begin{aligned}
		W_{\varphi_{\lambda}}u(t, x(t), \xi(t)) & = \exp\left( \int_{0}^{t} i P(x(\tau), \xi(\tau)) d\tau \right) W_{\varphi_{\lambda}}[u_{0}](x(0), \xi(0))                                \\
		                                        & \quad + i \int_{0}^{t} \exp\left( \int_{s}^{t} i P(x(\tau), \xi(\tau)) d\tau \right) Ru(s, x(s), \xi(s)) ds. \label{eq:integral_equation}
	\end{aligned}
\end{equation}
\if0
where $P$, $R_1$, and $R_2$ are given as follows:
\begin{align*}
	P(x(s), \xi(s))     & = -|\xi(s)|^{\theta} - V(x(s)) + (x(s) \cdot \nabla_{x}V(x(s))) , \\
	R_{1}(\eta, \xi(s)) & = \sum_{2 \le |\alpha| < L} \frac{1}{\alpha!} (\partial_{\xi}^{\alpha}|\xi|^{\theta})\Big|_{\xi=\xi(s)} (\eta - \xi(s))^{\alpha}, \\
	R_{2}(\eta, \xi(s)) & = L \sum_{|\alpha|=L} \frac{(\eta - \xi(s))^{\alpha}}{\alpha!} \int_{0}^{1} (1-\tau)^{L-1} (\partial_{\xi}^{\alpha}|\xi|^{\theta})\Big|_{\xi=\xi(s)+\tau(\eta-\xi(s))} d\tau.
\end{align*}
\fi

\section{Proof of Main Theorem}
\begin{proof}[Proof of Theorem \ref{thm:main}]
	We prove (ii) $\Rightarrow$ (i) of Theorem \ref{thm:main}. The converse follows from a similar argument.
	We fix an arbitrary $a \ge 1$, and let $V = K \times \Gamma$ be a neighborhood of $(x_0, \xi_0)$ satisfying condition (ii) of Theorem \ref{thm:main}.
	To prove Theorem \ref{thm:main}, it suffices to show that there exists a $\delta >0$ such that the following assertion $P(N)$ holds for all $N \in \N \cup \{0 \}$ .

	$P(N)$: For any $\varphi \in \mathcal{S}(\mathbb{R}^n) \setminus \{0\}$
	 there exists a constant $C_{N, a, \varphi} > 0$ such that
	\begin{equation*}
		|W_{\varphi_{\lambda}}u(s, x(s), \xi(s))| \le C_{N, a, \varphi}\lambda^{-\delta N}
	\end{equation*}
	holds for all $\lambda, x, \xi, s$ satisfying $\lambda \ge \lambda_0$, $a^{-1} \le |\xi| \le a$, $x \in K$, $\xi \in \Gamma$, and $T_0 \le s \le T$.

	We prove $P(N)$ by induction on $N$.
	First, $P(0)$ holds trivially from the boundedness of the wave packet transform, since $u \in C(\mathbb{R}; L^2(\mathbb{R}^n))$.
	Next, assuming that $P(N)$ holds, we show that $P(N+1)$ holds for some $\delta > 0$. We put $\sigma = \delta N$.
	For this purpose, it is sufficient to show that the remainder term $Ru$ in the second term on the right-hand side of the integral equation \eqref{eq:integral_equation} can be estimated as
	\begin{equation*}
		|Ru(s, x(s), \xi(s))| \le C_{N, a, \varphi}\lambda^{-(\sigma+\delta)} .
	\end{equation*}
	The remainder term is given by
	\begin{equation*}
		\begin{aligned}
			Ru(s, x(s), \xi(s)) & = -W_{\varphi_{\lambda}}[a(D)u(s, \cdot)](x(s), \xi(s))                                                                       \\
			                    & \quad - \iint \overline{\varphi_{\lambda}(y-x(s))} e^{iy\cdot(\eta-\xi(s))} R_{b}(\eta, \xi(s)) \widehat{u}(s, \eta) d\eta dy \\
			                    & \quad - W_{\varphi_{\lambda}}[R_{V}u(s, \cdot)](x(s), \xi(s)).
		\end{aligned}
	\end{equation*}
	We estimate each term using Lemma \ref{lem:estimates}. The degree $L$ of the Taylor expansion is taken as a sufficiently large natural number depending on $N$ in the proof.
	First, we estimate the term $W_{\varphi_{\lambda}}[a(D)u(s, \cdot)](x(s), \xi(s))$.
	By definition, we have
	\begin{align*}
	&	W_{\varphi_{\lambda}}[a(D)u(s, \cdot)](x(s), \xi(s)) \\
     &   = \int \overline{\varphi_{\lambda}(y-x(s))} (a(D)u(s, \cdot))(y) e^{-iy\cdot\xi(s)} dy.
	\end{align*}
	Using the relation $e^{-iy\cdot\xi(s)} = \frac{-\Delta_y}{|\xi(s)|^2}e^{-iy\cdot\xi(s)}$, from the integration by parts, we obtain
	\begin{align*}
	&	W_{\varphi_{\lambda}}[a(D)u(s, \cdot)](x(s), \xi(s)) \\
     &   = \int   \left( \frac{-\Delta_y}{|\xi(s)|^2} \right)^M \left( \overline{\varphi_{\lambda}(y-x(s))} (a(D)u(s, \cdot))(y) \right)e^{-iy\cdot\xi(s)} dy.
	\end{align*}
	By Lemma \ref{lem:estimates}, since $|\xi(s)| \ge C_2 \lambda$ holds for $\lambda \ge \lambda_0$, we have
	\begin{equation*}
		\frac{1}{|\xi(s)|^{2M}} \le C_2^{-2M}\lambda^{-2M}.
	\end{equation*}
	On the other hand,  From the boundedness of $ \partial^\alpha _x a(D)$ in $L^2$ and a direct computation of the scaling of $\varphi$, we have 
	\begin{equation*}
		\int |\Delta_y^M (\overline{\varphi_{\lambda}(y-x(s))} (a(D)u(s, \cdot))(y))| dy \le C_{M, \varphi} \lambda^{2Mb} \| u \|_{L^2}.
	\end{equation*}
	Consequently, it follows that
	\begin{equation*}
		|W_{\varphi_{\lambda}}[a(D)u(s, \cdot)]| \le (C'_{\xi})^{-2M} C_{M, \varphi} \lambda^{-2M(1-b)}.
	\end{equation*}
	Since $M$ can be chosen arbitrarily, by choosing $M$ sufficiently large such that $-2M < -(\sigma+\delta)$, we obtain
	\begin{equation*}
		|W_{\varphi_{\lambda}}[a(D)u(s, \cdot)]| \le (C'_{\xi})^{-2M} C_{M, \varphi} \lambda^{-(\sigma+\delta)}.
	\end{equation*}
	Next, we proceed to the estimate of $W_{\varphi_{\lambda}} [R_b(D)u(s, \cdot)](x(s), \xi(s))$.
	First, we estimate $W_{\varphi_{\lambda}} [R_{b,1}(D)u(s, \cdot)](x(s), \xi(s))$.
	Let $R_{b,\alpha} := (\partial_{\xi}^{\alpha}|\xi|^{\theta})|_{\xi=\xi(s)}(\eta-\xi(s))^{\alpha}$. Then,
	\begin{multline*}
W_{\varphi_{\lambda}}[R_{b,\alpha}(D)u(s, \cdot)](x(s), \xi(s))\\
= \int \overline{\varphi_{\lambda}(y-x(s))} e^{-iy\cdot\xi(s)}
\left( \int e^{iy\cdot\eta} R_{b,\alpha}(\eta, \xi(s)) \widehat{u}(s, \eta)\, d\eta \right) dy \\
= \iint \overline{\varphi_{\lambda}(y-x(s))}
\bigl((\partial_{\xi}^{\alpha}|\xi|^{\theta})\bigr)\big|_{\xi=\xi(s)}
(\eta-\xi(s))^{\alpha} \widehat{u}(s, \eta)
 e^{iy\cdot(\eta-\xi(s))}\, d\eta \, dy.
\end{multline*}
	Using the relation $(\eta-\xi(s))^{\alpha}e^{iy\cdot(\eta-\xi(s))} = (-i\nabla_{y})^{\alpha}e^{iy\cdot(\eta-\xi(s))}$, from the integration by parts with respect to $y$, we obtain
	\begin{multline*}
		W_{\varphi_{\lambda}}[R_{b,\alpha}(D)u(s, \cdot)](x(s), \xi(s))\\
		= \iint \overline{\varphi_{\lambda}(y-x(s))}
		\bigl((\partial_{\xi}^{\alpha}|\xi|^{\theta})\bigr)\big|_{\xi=\xi(s)} \widehat{u}(s, \eta)
		(-i\nabla_{y})^{\alpha} e^{iy\cdot(\eta-\xi(s))}\, d\eta\, dy \\
		= i^{|\alpha|} \bigl((\partial_{\xi}^{\alpha}|\xi|^{\theta})\bigr)\big|_{\xi=\xi(s)}
		\iint \nabla_{y}^{\alpha}(\overline{\varphi_{\lambda}(y-x(s))}) \widehat{u}(s, \eta)
		e^{iy\cdot(\eta-\xi(s))}\, d\eta\, dy \\
		= i^{|\alpha|} \bigl((\partial_{\xi}^{\alpha}|\xi|^{\theta})\bigr)\big|_{\xi=\xi(s)}
		W_{\nabla_{y}^{\alpha}\varphi_{\lambda}}[u(s, \cdot)](x(s), \xi(s)).
	\end{multline*}
	For the coefficient $(\partial_{\xi}^{\alpha}|\xi|^{\theta})|_{\xi=\xi(s)}$, from the estimates of the characteristic curves and the properties of homogeneous functions, there exists a constant $K_{\alpha} > 0$ such that
	\begin{equation*}
		|(\partial_{\xi}^{\alpha}|\xi|^{\theta})|_{\xi=\xi(s)}| \le K_{\alpha}\lambda^{\theta-|\alpha|}.
	\end{equation*}
	On the other hand, for $W_{\nabla_{y}^{\alpha}\varphi_{\lambda}}[u(s, \cdot)](x(s), \xi(s))$, since
	\begin{equation*}
		\nabla_{y}^{\alpha}\varphi_{\lambda}(y-x(s)) = \lambda^{\frac{nb}{2}} \nabla_{y}^{\alpha}\varphi(\lambda^{b}(y-x(s))) = \lambda^{\frac{nb}{2}}\lambda^{b|\alpha|}(\nabla_{y}^{\alpha}\varphi)(\lambda^{b}(y-x(s))),
	\end{equation*}
	putting $\psi(y) = \nabla_{y}^{\alpha}\varphi(y)$, we can write
	\begin{equation*}
		W_{\nabla_{y}^{\alpha}\varphi_{\lambda}}[u(s, \cdot)](x(s), \xi(s)) = \lambda^{b|\alpha|} W_{\psi_{\lambda}}[u(s, \cdot)](x(s), \xi(s)).
	\end{equation*}
applying the inductive hypothesis to $W_{\psi_{\lambda}}$, we have
	\begin{equation*}
		|W_{\nabla_{y}^{\alpha}\varphi_{\lambda}}[u(s, \cdot)](x(s), \xi(s))| \le C'_{\sigma, a, \psi}\lambda^{-\sigma+b|\alpha|}.
	\end{equation*}
	From the above,
	\begin{equation*}
		\begin{aligned}
			|W_{\varphi_{\lambda}}[R_{b,\alpha}(D)u(s, \cdot)](x(s), \xi(s))| & \le K_{\alpha} C_{N, a, \psi} \lambda^{\theta-|\alpha|} \lambda^{-\sigma+b|\alpha|} \\
			                                                                & = K_{\alpha} C_{N, a, \psi} \lambda^{-(\sigma-\theta+|\alpha|(1-b))}
		\end{aligned}
	\end{equation*}
	holds. Here, since $\sigma-\theta+|\alpha|(1-b) \ge \sigma-\theta+2(1-b)$, setting $\delta_{1} := -\theta+2(1-b)$, we get $|W_{\varphi_{\lambda}}[R_{b,\alpha}(D)u(s, \cdot)](x(s), \xi(s))| \le C''_{\sigma, a, \alpha, \varphi}\lambda^{-(\sigma+\delta_{1})}$, which implies that
	\begin{equation*}
		\begin{aligned}
			|W_{\varphi_{\lambda}}[R_{b,1}(D)u(s, \cdot)](x(s), \xi(s))| & \le \sum_{2 \le |\alpha| < L} \frac{1}{\alpha!} |W_{\varphi_{\lambda}}[R_{b,\alpha}(D)u(s, \cdot)](x(s), \xi(s))| \\
 & \le C_{N, a, \varphi} \lambda^{-(\sigma+\delta_{1})}.
		\end{aligned}
	\end{equation*}
	Next, we estimate $W_{\varphi_{\lambda}} [R_{b,2}(D)u(s, \cdot)](x(s), \xi(s))$.
	Using the definition of the (inverse) Fourier transform on $\mathcal{S}'$, we have
	\begin{equation*}
		\begin{aligned}
			W_{\varphi_{\lambda}}[R_{b,2}(D)u(s, \cdot)](x(s), \xi(s)) & = \int \overline{\varphi_{\lambda}(y-x(s))} R_{2}(D)u(s, y) e^{-iy\cdot\xi(s)} dy                                                       \\
			                                                         & = \int \mathcal{F}[R_{b,2}(D)u(s, \cdot)](\eta) \mathcal{F}^{-1}[\overline{\varphi_{\lambda}(\cdot-x(s))}e^{-i\cdot\xi(s)}](\eta) d\eta   \\
			                                                         & = \int R_{b,2}(\eta, \xi(s))\widehat{u}(s, \eta) \mathcal{F}^{-1}[\overline{\varphi_{\lambda}(\cdot-x(s))}e^{-i\cdot\xi(s)}](\eta) d\eta.
		\end{aligned}
	\end{equation*}
	We divide the integral into the following two regions $A_1$ and $A_2$:
	\begin{equation*}
		A_{1} := \{ \eta \in \mathbb{R}^{n} \mid |\eta-\xi(s)| \le \lambda^{c} \}, \quad A_{2} := \{ \eta \in \mathbb{R}^{n} \mid |\eta-\xi(s)| > \lambda^{c} \}.
	\end{equation*}
	Here, $c$ is chosen such that $b < c < 1$. 
\begin{align*}
  I_{1}
    &:= \int_{A_{1}} R_{b,2}(\eta,\xi(s))\,\widehat{u}(s,\eta)\,
        \mathcal{F}^{-1}\!\Bigl[\overline{\varphi_{\lambda}(\cdot-x(s))}\,
        e^{-i(\cdot)\cdot\xi(s)}\Bigr](\eta)\,d\eta,\\
  I_{2}
    &:= \int_{A_{2}} R_{b,2}(\eta,\xi(s))\,\widehat{u}(s,\eta)\,
        \mathcal{F}^{-1}\!\Bigl[\overline{\varphi_{\lambda}(\cdot-x(s))}\,
        e^{-i(\cdot)\cdot\xi(s)}\Bigr](\eta)\,d\eta,
\end{align*}
so that
\begin{equation*}
  W_{\varphi_{\lambda}}[R_{b,2}(D)u(s,\cdot)](x(s),\xi(s)) = I_{1} + I_{2}.
\end{equation*}

	First, we estimate the integral over the region $A_{1}$. We start by estimating $R_{2}(\eta, \xi(s))$. From Lemma \ref{lem:estimates}, when $\lambda \ge \lambda_{0}$, we have
	\begin{equation*}
		|\xi(s) + \tau(\eta-\xi(s))| \ge |\xi(s)| - \tau|\eta-\xi(s)| \ge \lambda(C_2 - \lambda^{c-1}).
	\end{equation*}
	Thus, taking $\lambda$ sufficiently large, we get $|\xi(s) + \tau(\eta-\xi(s))| \ge \frac{1}{2}C_2\lambda > 1$. Since the following inequality holds for $|\alpha|=L$
	\begin{equation*}
		|(\partial_{\xi}^{\alpha}|\xi|^{\theta})|_{\xi=\xi(s)+\tau(\eta-\xi(s))}| \le K_{L}|\xi(s)+\tau(\eta-\xi(s))|^{\theta-L},
	\end{equation*}
we have 
\begin{equation*}
		\begin{aligned}
			|R_{b,2}(\eta, \xi(s))| & \le L \sum_{|\alpha|=L} \frac{|\eta-\xi(s)|^{L}}{\alpha!} \int_{0}^{1} (1-\tau)^{L-1} |(\partial_{\xi}^{\alpha}|\xi|^{\theta})|_{\xi=\xi(s)+\tau(\eta-\xi(s))}| d\tau \\
			                      & \le L \sum_{|\alpha|=L} \frac{|\eta-\xi(s)|^{L}}{\alpha!} K_{L} \left( \frac{1}{2}C_{2}\lambda \right)^{\theta-L} \int_{0}^{1} (1-\tau)^{L-1} d\tau        \\
			                      & = \sum_{|\alpha|=L} \frac{K_{L}}{\alpha!} |\eta-\xi(s)|^{L} \left( \frac{1}{2}C_{2} \right)^{\theta-L} \lambda^{\theta-L}.
		\end{aligned}
	\end{equation*}
	Using this, $I_1$ can be estimated as
	\begin{equation*}
		\begin{aligned}
			|I_{1}| & \le \int_{A_{1}} |R_{2}(\eta, \xi(s))| |\widehat{u}(s, \eta)| |\mathcal{F}^{-1}[\overline{\varphi_{\lambda}(\cdot-x(s))}e^{-i\cdot\xi(s)}](\eta)| d\eta                                                                                                       \\
			        & \le \sum_{|\alpha|=L} \frac{K_{L}}{\alpha!} \left( \frac{1}{2}C_{2} \right)^{\theta-L} \lambda^{\theta-L(1-c)} \int_{A_{1}} |\widehat{u}(s, \eta)| \cdot |\mathcal{F}^{-1}[\overline{\varphi_{\lambda}(\cdot-x(s))}e^{-i\cdot\xi(s)}](\eta)| d\eta \\
			        & \le \sum_{|\alpha|=L} \frac{K_{L}}{\alpha!} \left( \frac{1}{2}C_{2} \right)^{\theta-L} \|\varphi\|_{L^2} \|u_{0}\|_{L^2} \lambda^{\theta-L(1-c)}.
		\end{aligned}
	\end{equation*}
	Noting that $1-c > 0$ and taking $L$ sufficiently large, there exists a constant $C_{\sigma, \varphi} > 0$ such that
	\begin{equation*}
		|I_{1}| \le C_{\sigma, \varphi}\lambda^{-(\sigma+\delta_{1})}.
	\end{equation*}
	Next, we estimate the integral over the region $A_{2}$.
	\begin{equation*}
		\begin{aligned}
			|I_{2}| & \le  \int_{A_{2}} \frac{|\eta-\xi(s)|^{M}}{|\eta-\xi(s)|^{M}} |R_{b,2}(\eta, \xi(s))| |\widehat{u}(s, \eta)| \ |\mathcal{F}^{-1}[\overline{\varphi_{\lambda}(\cdot-x(s))}e^{-i\cdot\xi(s)}](\eta)| d\eta \\
			        & \le \lambda^{-Mc} \int_{A_{2}} |R_{b,2}(\eta, \xi(s))| |\widehat{u}(s, \eta)| \cdot |\eta-\xi(s)|^{M} |\mathcal{F}^{-1}[\overline{\varphi_{\lambda}(\cdot-x(s))}e^{-i\cdot\xi(s)}](\eta)| d\eta.
		\end{aligned}
	\end{equation*}
	Here, setting $\gamma := \xi(s) + \tau(\eta-\xi(s))$, we check the boundedness of $(\partial_{\xi}^{\alpha}b(\xi))|_{\xi=\gamma}$ for $|\alpha|=L$.
	When $|\gamma| \le \frac{C_2\lambda_{0}}{2}$, $(\partial_{\xi}^{\alpha}b(\xi))|_{\xi=\gamma}=0$ from the presence of the cut-off function $\chi$.
	When $|\gamma| > \frac{C_2\lambda_{0}}{2}$, $(\partial_{\xi}^{\alpha}b(\xi))|_{\xi=\gamma}$ is obviously bounded for large $|\alpha |=L$.
	Thus, there exists a constant $C_{N, a, L}^{\prime} > 0$ such that $|R_{2}(\eta, \xi(s))| \le C_{N, a, L}^{\prime}|\eta-\xi(s)|^{L}$, from which, we have from the Cauchy-Schwarz inequality that
\begin{equation*}
\begin{split}
|I_{2}|
\le \lambda^{-Mc} C_{N,a,L}^{\prime}
\int_{A_{2}} |\widehat{u}(s,\eta)|\,|\eta-\xi(s)|^{M+L}
\left|\mathcal{F}^{-1}\bigl[\overline{\varphi_{\lambda}(\cdot-x(s))}e^{-i\cdot\xi(s)}\bigr](\eta)\right|\,d\eta \\
\le \lambda^{-Mc} C_{N,a,L}^{\prime}\,\|u_{0}\|_{L^2}
\left(
\int_{\mathbb{R}^n} |\eta-\xi(s)|^{2(M+L)}
\left|\mathcal{F}^{-1}\bigl[\overline{\varphi_{\lambda}(\cdot-x(s))}e^{-i\cdot\xi(s)}\bigr](\eta)\right|^{2}
\,d\eta
\right)^{1/2}.
\end{split}
\end{equation*}
from  the scaling property and the change of variable, we have
	\begin{equation*}
		\begin{aligned}
			\int_{\mathbb{R}^n} |\eta-\xi(s)|^{2(M+L)} & |\mathcal{F}^{-1}[\overline{\varphi_{\lambda}(\cdot-x(s))}e^{-i\cdot\xi(s)}](\eta)|^{2} d\eta                                           \\
			                                           & = \int_{\mathbb{R}^n} |\eta-\xi(s)|^{2(M+L)} |e^{ix(s)\cdot\eta} \overline{\mathcal{F}^{-1}[\varphi_{\lambda}](\eta-\xi(s))}|^{2} d\eta \\
			                                           & = \int_{\mathbb{R}^n} |\zeta|^{2(M+L)} |\mathcal{F}^{-1}[\varphi_{\lambda}](\zeta)|^{2} d\zeta                                          \\
			                                           & = \lambda^{2b(M+L)} \int_{\mathbb{R}^n} |\rho|^{2(M+L)} |\mathcal{F}^{-1}[\varphi](\rho)|^{2} d\rho                                     \\
			                                           & = C_{N, M, \varphi}^{\prime} \lambda^{2b(M+L)}.
		\end{aligned}
	\end{equation*}
Therefore we have that
	\begin{equation*}
		|I_{2}| \le \lambda^{-Mc} C_{N, a, L}^{\prime} \|u_{0}\|_{L^2} (C_{N, M, \varphi}^{\prime} \lambda^{2b(M+L)})^{1/2} = C_{N, a, L, M, \varphi} \lambda^{-M(c-b)+bL}.
	\end{equation*}
Taking $M$ sufficiently large such that $-M(c-b)+bL < -\sigma - \delta_1$ (note that $c-b > 0$), we obtain
	\begin{equation*}
		|I_{2}| \le C_{N, a, \varphi}\lambda^{-\sigma-\delta_1}.
	\end{equation*}
	From the above arguments, for $W_{\varphi_{\lambda}}[R_{2,b}(D)u(s, \cdot)](x(s), \xi(s))$, there exists a constant $C_{\sigma, a, \varphi} > 0$ such that
	\begin{equation*}
		|W_{\varphi_{\lambda}}[R_{2}(D)u(s, \cdot)](x(s), \xi(s))| \le C_{N, a, \varphi}\lambda^{-(\sigma+\delta_{1})}.
	\end{equation*}
	Finally, we estimate $W_{\varphi_{\lambda}} [R_V u(s, \cdot)](x(s), \xi(s))$ with $R_{V}(y, x) = R_{V,1}(y, x) + R_{V,2}(y, x)$. The strategy of this estimate is entirely based on Kato, Kobayashi Ito \cite{KatoKobayashiIto2017}.
	First, we estimate  $I_{\alpha}(s)$ corresponding to each term of $R_{V,1}$ defined as follows:
	\begin{equation*}
		\begin{aligned}
			I_{\alpha}(s) & := W_{\varphi_{\lambda}}[(\partial_{x}^{\alpha}V)(x(s))(\cdot-x(s))^{\alpha}u(s, \cdot)](x(s), \xi(s))                      \\
			              & = (\partial_{x}^{\alpha}V)(x(s)) \int \overline{\varphi_{\lambda}(y-x(s))} (y-x(s))^{\alpha} u(s, y) e^{-iy\cdot\xi(s)} dy.
		\end{aligned}
	\end{equation*}
	Let $\psi(y) = y^{\alpha}\varphi(y)$.
	Using the relation  that 
    \[
    (y-x(s))^{\alpha}\overline{\varphi_{\lambda}(y-x(s))} = \lambda^{-b|\alpha|}\overline{(y^{\alpha}\varphi)_{\lambda}(y-x(s))}= \lambda^{-b|\alpha|} \psi_\lambda (y-x(s)),
    \] we have that
	\begin{equation*}
		W_{\varphi_{\lambda}}[(\cdot-x(s))^{\alpha}u(s, \cdot)](x(s), \xi(s)) = \lambda^{-b|\alpha|} W_{\psi_{\lambda}}[u(s, \cdot)](x(s), \xi(s)).
	\end{equation*}
	Using the inductive hypothesis, we have $|W_{\varphi_{\lambda}}[(\cdot-x(s))^{\alpha}u(s, \cdot)](x(s), \xi(s))| \le C_{N, a, \psi}\lambda^{-\sigma-b|\alpha|}$.
	From the assumption on the potential, $|(\partial_{x}^{\alpha}V)(x(s))| \le C_{\alpha}\langle x(s)\rangle^{\nu-|\alpha|}$ holds. Combining these estimates with Lemma \ref{lem:estimates}, in the case that $1 < \theta < 2$ and $|\nu| \geq |\alpha|$, 
    we have 
\begin{align*}
|I_{\alpha}(s)|
&\le C_{\alpha}C_{\sigma, a, \psi}\lambda^{-(\sigma+2b)} \\
&\le  C_{N, a, \psi}\lambda^{-\sigma-b|\alpha|+(\theta-1)(\nu-|\alpha|)}\\
& \le  C_{N, a, \psi}\lambda^{-\sigma-2b+(\theta-1)(\nu-2)}.
\end{align*}
	From the assumption on $b$ in Theorem \ref{thm:main}, $-2b+(\theta-1)(\nu-2) < 0$ holds. Thus, the desired decay $|I_{\alpha}(s)|$ for $\lambda$ is obtained.
    In the case that $1 < \theta < 2$ and $|\nu| \leq  |\alpha|$
 from $\langle x(s)\rangle^{\nu-|\alpha|} \le 1$, we can see that
\begin{align*}
|I_{\alpha}(s)| \leq \lambda^{-\sigma-2b}.
\end{align*}
Hence, we have the desired decay $|I_{\alpha}(s)|$ with $1 < \theta < 2$.
	If $0 < \theta \le 1$, by Lemma \ref{lem:estimates}, $|x(s)|$ is bounded with $\lambda$.
	Thus, there exists a constant $C'_{\sigma, a, \varphi} > 0$ such that
	\begin{equation*}
		|I_{\alpha}(s)| \le C'_{\sigma, a, \varphi}\lambda^{-(\sigma+2b)}.
	\end{equation*}
Putting $\delta_{2} := \min(2b, 2b-(\theta-1)(\nu-2))$, we have 
	\begin{equation*}
		|W_{\varphi_{\lambda}}[R_{V,1}u(s, \cdot)](x(s), \xi(s))| \le C_{N, a, \varphi}\lambda^{-(\sigma+\delta_{2})}.
	\end{equation*}
Next, we estimate $R_{V,2}$, which can be written as
\begin{equation*}
W_{\varphi_{\lambda}}[R_{V,2}u(s, \cdot)](x(s), \xi(s))
= \int \overline{\varphi_{\lambda}(y-x(s))} R_{V,2}(y, x(s)) u(s, y) e^{-iy\cdot\xi(s)} \, dy.
\end{equation*}
We decompose $\mathbb{R}^n = B_1 \cup B_2$ with
\begin{equation*}
B_{1} := \{ y \in \mathbb{R}^{n} \mid |y-x(s)| \le \lambda^{-d} \}, \quad
B_{2} := \{ y \in \mathbb{R}^{n} \mid |y-x(s)| > \lambda^{-d} \},
\end{equation*}
where $0 < d < b$, and set
\begin{align*}
J_{1} &:= \int_{B_{1}} \overline{\varphi_{\lambda}(y-x(s))} R_{V,2}(y, x(s)) u(s, y) e^{-iy\cdot\xi(s)} \, dy,\\
J_{2} &:= \int_{B_{2}} \overline{\varphi_{\lambda}(y-x(s))} R_{V,2}(y, x(s)) u(s, y) e^{-iy\cdot\xi(s)} \, dy.
\end{align*}
	First, we estimate $J_1$. From the Cauchy-Schwarz inequality, we have
	\begin{equation*}
		\begin{aligned}
			|J_{1}(s)| & \le \left( \int_{B_{1}} |\overline{\varphi_{\lambda}(y-x(s))} R_{V,2}(y, x(s))|^{2} dy \right)^{\frac{1}{2}} \left( \int_{B_{1}} |u(s, y) e^{-iy\cdot\xi(s)}|^{2} dy \right)^{\frac{1}{2}} \\
			           & \le \|u_{0}\|_{L^2} |B_{1}|^{\frac{1}{2}} \sup_{y \in B_{1}} |\varphi_{\lambda}(y-x(s))| \sup_{y \in B_{1}} |R_{V,2}(y, x(s))|.
		\end{aligned}
	\end{equation*}
	Here, $|B_{1}| = C_{n}\lambda^{-nd}$. For $\varphi_{\lambda}$, we have $\sup_{y \in B_{1}} |\varphi_{\lambda}(y-x(s))| \le C_{\varphi}\lambda^{\frac{nb}{2}}$.
	Next we estimate $R_{V,2}(y, x(s))$. Let $z := x(s)+\tau(y-x(s))$. From the assumption on the potential $V$, $|(\partial_{x}^{\alpha}V)(z)| \le C_{L}\langle z \rangle^{\nu-L}$.
	Since $L$ is large enough such that $\nu \le L$, we have $\langle z \rangle^{\nu-L} \le 1$. Thus, $|R_{V,2}(y, x(s))| \le C_L \lambda^{-Ld}$.
	Combining these estimates, we have for sufficiently large $L$
	\begin{align*}
		|J_1(s)| & \le \|u_0\| C_n \la^{-nd} \sqrt{C_\vp} \la^{nb/2} \left( \sum_{|\alpha|=L}\frac{1}{\alpha!} C_L \la^{-Ld} \right) \\
		         & = C_{L,\vp} \la^{-nd/2 + nb/2 - Ld}                                                                               \\
		         & = C_{\sigma,\vp}\la^{-(\sigma+\delta_1)}.
	\end{align*}
	Next, we estimate the integral over $B_{2}$.
	\begin{equation*}
		|J_{2}(s)| \le \left( \int_{B_{2}} |\overline{\varphi_{\lambda}(y-x(s))} R_{V,2}(y, x(s))|^{2} dy \right)^{\frac{1}{2}} \|u_{0}\|_{L^2}.
	\end{equation*}
Similar to  the estimate for $B_{1}$, we have 
	\begin{equation*}
		\begin{aligned}
			K_{2}(s) & \le C_{L}^{2} \int_{B_{2}} |\varphi_{\lambda}(y-x(s))|^{2} |y-x(s)|^{2L} dy                                \\
			         & = C_{L}^{2} \lambda^{nb-2bL} \int_{B_{2}} |\varphi(\lambda^{b}(y-x(s)))|^{2} |\lambda^{b}(y-x(s))|^{2L} dy \\
			         & = C_{L}^{2} \lambda^{-2bL} \int_{|z| > \lambda^{b-d}} |\varphi(z)|^{2} |z|^{2L} dz                         \\
			         & \le C_{L, \varphi} \lambda^{-2bL}.
		\end{aligned}
	\end{equation*}
	Using this, we obtain
	\begin{equation*}
		|J_{2}(s)| \le \|u_{0}\|_{L^2} \sqrt{C_{L, \varphi}} \lambda^{-bL}.
	\end{equation*}
	Taking $L$ sufficiently large yields the desired decay.

	To summarize, let $\delta := \min\{\delta_{1}, \delta_{2}\}$. Then there exists a constant $C_{\sigma, a, \varphi} > 0$ such that
	\begin{equation*}
		|R_{u}(s, x(s), \xi(s))| \le C_{N, a, \varphi}\lambda^{-(\sigma+\delta)}.
	\end{equation*}
	Thus, we obtain the desired estimate for the remainder term, and the assertion $P(N)$ is proved for all $\sigma$ and $\varphi$ by induction.
	This completes the proof of Theorem \ref{thm:main}.
\end{proof}

\section{Proof of Corollaries}
\subsection{Proof of Corollary \ref{wfs:cor:main_no_potential}}
The proof proceeds in essentially the same way as that of the main theorem. 
	From \eqref{wfs:eq:characteristics_ode_no_potential}, $\tilde{x}(s):=\tilde{x}(s;t,x,\lambda\xi)$ and $\tilde{\xi}(s):=\tilde{\xi}(s;t,x,\lambda\xi)$ are expressed as follows:
	\begin{equation*}
		\begin{cases}
			\tilde{x}(s) = x + (s-t)\theta \lambda^{\theta-1}|\xi|^{\theta-2}\xi, \\
			\tilde{\xi}(s) = \lambda\xi.
		\end{cases}
	\end{equation*}
    Of course, these characteristic curves also satisfy the assertion of Lemma \ref{lem:estimates}. 
Due to this modification of the characteristic curves, the following term appears as an additional remainder:
	\begin{equation*}
		\nabla_x V(\tilde{x}(s)) \int (y-\tilde{x}(s)) \overline{\varphi_{\lambda}(y-\tilde{x}(s))} u(y) e^{-iy\cdot\tilde{\xi}(s)} dy.
	\end{equation*}
All remainder terms, other than this one, are estimated in exactly the same way as in the proof of the main theorem.
We only consider the estimate for this term. As in the previous computations, setting $\psi(y)=y\varphi(y)$, this remainder term can be expressed as
	\begin{equation*}
		\nabla_x V(\tilde{x}(s)) \lambda^{-b} W_{\psi_\lambda}[u](s,\tilde{x}(s),\tilde{\xi}(s)).
	\end{equation*}
Applying  Lemma \ref{lem:estimates} with $1<\theta<2$, from the assumption on $V$ in Corollary \ref{wfs:cor:main_no_potential}, it follows that
	\begin{equation*}
		|\nabla_x V(\tilde{x}(s))| \le C'\lambda^{(\theta-1)(\nu-1)}.
	\end{equation*}
	Also, noting that we can take $b$ such that $(\theta-1)(\nu-1)<b<(2-\theta)/2$  under the assumption that $\nu <\theta/2(\theta-1)$, we have 
	\begin{equation*}
		|\nabla_x V(\tilde{x}(s)) \lambda^{-b} W_{\psi_\lambda}[u](s,\tilde{x}(s),\tilde{\xi}(s))| \le C_{\sigma,a,\varphi} \lambda^{-(\sigma+\delta)},
	\end{equation*}
where $\delta = -b+(\theta-1)(\nu-1)$.
The rest of the proof is identical to that of  Theorem \ref{thm:main}, and this completes the proof of the corollary \ref{wfs:cor:main_no_potential}.

\subsection{Proof of Corollary \ref{wfs:cor:main_time_propagation} with $\theta=1$}
The assertion  of Corollary \ref{wfs:cor:main_no_potential} is also true for the case $\theta=1$.
Moreover, when $\theta=1$, the characteristic curve in the $x$-space is independent of $\lambda$, and therefore Corollary \ref{wfs:cor:main_time_propagation} follows directly from Proposition \ref{prop:wfs_characterization}.

\subsection{Proof of Corollary \ref{wfs:cor:main_time_propagation} with $0< \theta <1$}
It suffices to carry out the same argument as in the proof of Theorem \ref{thm:main} along the  trivial characteristic curves  $(\tilde{x}(s),\tilde{\xi}(s))=(x,\lambda \xi)$.
As a consequence of introducing these characteristic curves, the following additional remainder term arises:
\begin{align*}
\nabla_{\xi}b(\xi) \cdot \lambda^b W_{(\nabla \varphi)_{\lambda}} [u](x, \xi).
\end{align*}
By Lemma \ref{lem:estimates} and the induction hypothesis, we obtain that
\begin{align*}
|\nabla_{\xi}b(\xi) \cdot \lambda^b W_{(\nabla \varphi_{\lambda})} [u](x, \lambda \xi)| \leq C_{N, a, \varphi} \lambda^{-\sigma-1+\theta+b}.  
\end{align*}
Taking $b$  sufficiently small such that $-1 +\theta +b < 0$, we obtain the desired decay for this remainder term.
Therefore, we complete the proof of Corollary \ref{wfs:cor:main_time_propagation}.

\section*{Acknowledgements}
T.K. was supported by the Japan Student Services Organization (JASSO) Scholarship for Graduate Students (First Category).
Y.S. was partially supported by Grants-in-Aid for Scientific Research (C) (No. 23K03169).

\bibliographystyle{amsplain}
\bibliography{reference}

\end{document}